\documentclass[12pt]{article}
\pdfoutput=1
\usepackage{amsmath,amsthm,amssymb,bbm,color,enumerate}
\usepackage{graphicx}
\usepackage[UKenglish]{babel}
\usepackage{lmodern}
\usepackage{microtype}
\usepackage{mathtools}
\usepackage{esint}
\usepackage{comment,color}
\usepackage{hyperref}
\usepackage{authblk}
\usepackage[utf8]{inputenc}
\usepackage{float}
\usepackage[top=20mm, bottom=20mm, left=20mm, right=20mm]{geometry}
\usepackage{enumitem}

\mathtoolsset{showonlyrefs}	

\theoremstyle{plain}
\newtheorem{theorem}{Theorem}

\newtheorem{lemma}[theorem]{Lemma}

\newtheorem{proposition}[theorem]{Proposition}

\theoremstyle{definition}

\theoremstyle{remark}
\newtheorem{remark}[theorem]{Remark}

\newcommand*{\R}{\mathbb{R}}

\newcommand*{\eps}{\varepsilon}
\newcommand*{\doo}{\partial}

\newcommand{\parens}[1]{\left( #1 \right)}
\newcommand{\joukko}[1]{\left\{ #1 \right\}}
\newcommand{\abs}[1]{\left\lvert #1 \right\rvert}

\newcommand{\norm}[1]{\left\| #1 \right\|}
\newcommand{\der}{\mathrm{d}}

\title{Inverse problems for a model of biofilm growth}

\author[1,2,3]{Tommi Brander}
\author[4]{Daniel Lesnic}
\author[5]{Kai Cao}

\affil[1]{University of South-Eastern Norway, Department for Mathematics and Science Education, Norway}
\affil[2]{Norwegian University of Science and Technology, Department for Mathematical Sciences, Norway}
\affil[3]{Technical University of Denmark, Department of Applied Mathematics and Computer Science, Denmark}
\affil[4]{University of Leeds, Department of Applied Mathematics, UK}
\affil[5]{Southeast University, Department of Mathematics, Nanjing, China}

\begin{document}

\maketitle

\begin{abstract}
A bacterial biofilm is an aggregate of micro-organisms growing fixed onto a solid surface, rather than floating freely in a liquid. Biofilms play a major role in various practical situations such as surgical infections and water treatment. We consider a non-linear PDE model of biofilm growth subject to initial and Dirichlet boundary conditions, and the inverse coefficient problem of recovering the unknown parameters in the model from extra measurements of quantities related to the biofilm and substrate. By addressing and analysing this inverse problem we provide reliable and robust reconstructions of the 
primary physical quantities of interest represented by the diffusion coefficients of substrate and biofilm, the biomass spreading parameters, the maximum specific consumption and growth rates, the biofilm decay rate and the half saturation constant. 
We give
particular attention to the constant coefficients involved in the leading-part non-linearity, and present a uniqueness proof and some numerical results. In the course of the numerical investigation, we have identified extra data 
information that enables improving the reconstruction of the eight-parameter set of physical quantities associated to the model of biofilm growth. 
\end{abstract}

%






\paragraph{Email:} tommi.brander@usn.no, kcao@seu.edu.cn, d.lesnic@leeds.ac.uk

\paragraph{Keywords:} Biofilm; inverse problem; uniqueness; parameter estimation; reaction-diffusion system; degenerate parabolic system

\paragraph{MSC:} 
35R30, 
35Q92, 
65M32, 
35K40, 
35K59, 
35K65, 
35K67, 
65M06, 
65Z05, 
35B65 


\section{Introduction}
Communities of microbial cells create biofilms that are encountered in various processes related to plant growth promotion and protection, sewage bio-remediation, chronic infections and industrial bio-fouling \cite{Wallace:Li:Davidson:2016}. One of the typical consequences of biofilm formation is that resident microbes become significantly resistant to physical stresses and 
anti-microbial agents. It is therefore very important to model the biofilm growth for characterisation, modelling and control. In this spirit, in this paper 
we consider the inverse problem of recovering the constant parameters in a reaction-diffusion model of biofilm growth.
Biofilms are created when bacteria form a highly resilient matrix, rather than float freely.
For more on biofilms, see reviews such as \cite{Mazza:2016,Wilson:Lukowicz:Merchant:ValquierFlynn:Caballero:Sandoval:Okuom:Huber:Brooks:Wilson:Clement:Wentworth:Holmes:2017}.

Biofilm growth can be modelled by a system of coupled partial differential equations with initial and boundary conditions, 
see system~\eqref{eq:biofilm} below, originally due to Eberl, Parker and Van Loosdrecht~\cite{Eberl:Parker:VanLoosdrecht:2001} and later investigated by Efendiev, Zelik and Eberl~\cite{Efendiev:Zelik:Eberl:2009}.
The equations have parameters that relate to the growth rate, the death rate, and the spreading of the biofilm, as well as the behaviour of nutrients the biofilm feeds on.
We identify these parameters from information that could be provided by measurements of physical quantities related to the biofilm and the substrate. For numerical reconstruction we use a nonlinear least-squares solver, where we minimize the discrepancy between the observed density flux of nutrients and the computed solution produced by the parameters we are optimizing for.

The main challenges are the nonlinear nature of the mathematical problem and recovering the parameters that are involved in the nonlinearity.

Consider a biofilm whose density $0 \le M(x,t) < 1$ is a measurable function, and a substrate whose density $0 \le S(x,t) \le 1$ is also measurable for $(x,t) \in \Omega \times \mathbb{R}_{+}$, where $\Omega \subset \mathbb{R}^{n}$ is a Lipschitz bounded domain with piecewise smooth boundary $\partial \Omega$.  
Here, substrate density means the density of nutrients; in the literature on biofilms, substrate can also mean the material the biofilm is growing on and attached to, but we do not use the word in this sense.
The following pair of nonlinear parabolic Lotka-Volterra-type equations provides a continuum model for the growth of biofilms~\cite[section~5.1]{Efendiev:2013}:
\begin{eqnarray} \label{eq:biofilm}
\begin{cases}
\doo_t S = d_1 \Delta_x S - K_1 \frac{SM}{K_4 + S} + F(x,t), \quad (x,t) \in \Omega \times \R_{+}, \\
\doo_t M = d_2 \nabla_x \cdot \parens{\frac{M^b}{\parens{1-M}^a} \nabla_x M} - K_2 M + K_3 \frac{SM}{K_4+S} + G(x,t), \quad (x,t) \in \Omega \times \R_{+}, \\
S|_{\doo \Omega \times \R_+} = 1, \quad M|_{\doo \Omega \times \R_+} = 0, \\
S|_{t=0} = S_0, \quad M|_{t=0} = M_0, \end{cases} 
\end{eqnarray}
with biological constants $d_1 > 0$, $d_2 > 0$, $K_1 \geq 0$, $K_2 \geq 0$, $K_3 \geq 0$, $K_4 > 0$, $a \geq 0$ and $b \geq 1$, where $F$ and $G$ are given source functions.
For simplicity, we assume $K_1 > 0$ and $K_3 > 0$; otherwise, at least one of the two partial differential equations decouples.
Injectivity proofs and numerical reconstructions can also be established in the decoupled case with the same techniques of this paper. The physical meaning of the quantities present in the mathematical model \eqref{eq:biofilm} are as follows:
\begin{itemize}
\item $d_1$: substrate diffusion coefficient
\item $d_2$: biofilm diffusion coefficient
\item $K_1$: maximum specific consumption rate
\item $K_2$: biofilm decay rate
\item $K_3$: maximum specific growth rate
\item $K_4$: Monod's half saturation constant~\cite[page 383]{Monod:1949}
\item $a, b$: biomass spreading parameters
\item $S_{0}$ and $M_{0}$ are the initial densities at time $t=0$ of the substrate and biofilm, respectively, satisfying the compatibility conditions $S_0|_{\doo \Omega}=1$ and $M_0|_{\doo \Omega}=0$ with the Dirichlet boundary data. 
Zero Neumann insulated boundary condition on the density $M$ may also be prescribed instead of the zero Dirichlet boundary condition.
\end{itemize}
If $a=b=0$, then the system~\eqref{eq:biofilm} yields a well-investigated semilinear two-species predator-prey Lotka-Volterra model, but with the assumptions $a>0$ and $b \ge 1$, the nonlinear diffusivity
$\lambda(M) = \frac{M^b}{\parens{1-M}^a}$ makes the quasilinear biofilm model-problem challenging.

Although biofilms are heterogeneous~\cite{Costerton:Stewart:Greenberg:1999,Mazza:2016}, the model we use treats them as homogeneous masses.
We deem this necessary to keep the model sufficiently simple.
The inverse problem approach allows fitting the model to actual biofilms and thereby checking its validity. 
Other possible developments include more complicated models with several species of biofilms or antibiotics.
Alternatively, one might wish to consider more elaborate models of nutrient flow, rather than the simple diffusion we have used here.

Returning to the model at hand, we remark that if all the eight biological constants contained in the vector $\underline X = \parens{d_{1}, d_{2}, K_1, K_2, K_3, K_4, a, b}$ are known, then the direct problem has a unique solution $(S(x,t),M(x,t))$.
This is also true for mixed Dirichlet-Neumann boundary conditions, but pure Neumann boundary conditions may cause problems~\cite[section~5.1]{Efendiev:2013}.

Now, suppose we can observe the densities of the biofilm~$M$ and the substrate~$S$ over some suitable set of space-time points,  and that the coefficients in the governing PDEs are unknown.
Can all or some of the eight coefficients above be uniquely recovered from such observations?

First, we note a trivial obstruction:
\begin{remark}
In the homogeneous case $F=G=0$ in \eqref{eq:biofilm}, if $S_{0} \equiv 1$ and $M_{0} \equiv 0$, then it follows that $S \equiv 1$ and $M \equiv 0$ form a trivial solution of the system~\eqref{eq:biofilm} and none of the coefficients can be recovered.
Even if only $M \equiv 0$, the most interesting parameters $a$, $b$ and $d_2$ are unrecoverable, as there is no biofilm to observe.
We assume tacitly throughout that this is not the case.
\end{remark}

Next, we consider special situations which lead to uniqueness. These are similar in spirit to the critical couples used by Lorz, Pietschmann and Schlottbom~\cite{Lorz:Pietschmann:Schlottbom:2019}. When investigating uniqueness we assume no noise in the input data. On the other hand, noisy data are necessary to be considered when investigating the stability of the solution.

\begin{theorem}
\label{thm:unique}
Assume that $S$ and $M$ are known everywhere and that there exist special points with properties given by the assumptions below.
Then, all the 8 coefficients $d_1$, $d_2$, $K_1$, $K_2$, $K_3$, $K_4$, $a$ and $b$ can be uniquely determined.
The assumptions are:
\begin{enumerate}[label=(\roman*)]
\item There exists a point $(x_0,t_0) \notin \doo \joukko{(x, t) \in \Omega \times \R_+ ; M(x, t) = 0}$ where $\Delta_x S(x_0, t_0) \neq 0$ and either $S(x_0,t_0) = 0$ or $S(x_0,t_0) > 0$ and 
$M(x_0,t_0) = 0$. \label{ass:iii}
\item There exist points $(x_j,t_j)$, $j \in \joukko{1,2}$, such that the vectors
\begin{equation}
\parens{\doo_{t} S(x_j, t_j) - d_{1} \Delta_{x} S(x_j, t_j) - F(x_{j},t_{j}), S(x_j, t_j)M(x_j, t_j)}, \quad j=1,2, \nonumber
\end{equation}
are linearly independent.
\label{ass:iv}
\item There exist times $t_j$, $j \in \joukko{3,4}$, such that the vectors
	\begin{equation}
\label{eq:iii_li}
\parens{\int_{\Omega}M(x,t_j) \der x , \int_{\Omega} \frac{S(x, t_j) M(x, t_j)}{K_4+S(x, t_j)} \der x }, \quad j=3,4,
	\end{equation}
are linearly independent. \label{ass:v}
\item There exist points $(x_j,t_j)$, $j \in \joukko{5,6,7}$,  
with $M(x_j,t_j)$ taking different values, and for all $j \in \joukko{5, 6, 7}$ it holds that $\Delta_x M(x_j,t_j) \neq 0$, $\nabla_x M(x_j,t_j) = 0$ and $0 < M(x_j,t_j) < 1$. \label{ass:vi}
\end{enumerate}
\end{theorem}
For a proof of the theorem, see section~\ref{sec:unique}.
The derivatives exist due to Lemma~\ref{lemma:regularity} in Appendix \ref{AppendixA}. Below we give some justifications as to why the above assumptions $(i)$-$(iv)$ are sensible and practically realistic.\\
$\bullet$ A typical situation for biofilm is that it starts from a configuration where it covers only part of the domain~$\Omega$, that is, the support of $M_0$ is compactly contained in $\Omega$.
In such a case, the numerical results~\cite[section~5.1, numerical example (d)]{Efendiev:2013} suggest that the biofilm density will continue to have compact support for an extended period of time.
Further, many other equations of porous medium type do enjoy finite speed of propagation~\cite[section~1.2.1]{Vazquez:2007}.

Outside of the support of $M$, the substrate density satisfies a diffusion equation,
and there is no reason to assume this density to be a stationary harmonic function.\\
$\bullet$ There is no reason for the substrate density to be constant.\\
$\bullet$ There is no reason to assume that the biofilm mass and the biofilm feeding rate remain linearly coupled in a nonlinear model.\\
$\bullet$ If the biofilm density has a strict local maximum that changes value over time, or if the biofilm density has several strict local maxima at any time, condition $(iv)$ is satisfied. Both of these are reasonable scenarios to assume.

Finally, we remark that Theorem \ref{thm:unique} provides the unique retrieval of the 8 coefficients forming the vector of unknown parameters $\underline{X} \in \mathbb{R}^{8}$, but this recovery does not guarantee that the coefficients belong to the physically admissible set
\begin{eqnarray}
\underline{X} =  \parens{d_1, d_2, K_1, K_2, K_3, K_4, a, b} \in \frak{X} := (0,\infty)^{3} \times [0, \infty) \times (0, \infty)^2 \times [0,\infty) \times [1, \infty) \subset \R^{8}. \nonumber
\end{eqnarray}
Further discussions on this point are made in section~\ref{sec:unique}.


\subsection{Imaging and measuring biofilms}
\label{sec:imaging}
Both optical coherence tomography~\cite{Xi:Marks:Schlachter:Luo:Boppart:2006,Wagner:Horn:2017} and confocal scanning laser microscopy (CSLM)~\cite{Lawrence:Korber:Hoyle:Costerton:Caldwell:1991,Schlafer:Meyer:2017} \cite[section~2.2.5]{Surman:Walker:Goddard:Morton:Keevil:Weaver:Skinner:Hanson:Caldwell:Kurtz:1996} have been used to effectively observe the fine structure of biofilms in a non-destructive way.
Optical measurements have also been used to measure biofilm thickness~\cite{Bakke:Kommedal:Kalvenes:2001}, even continuously in time~\cite{Milfersted:Pons:Morgenroth:2006}.
Different techniques are reviewed in the articles~\cite{Azeredo:Azevedo:Briandet:Cerca:Coenye:Costa:Desvaux:DiBonaventura:Hebraud:Jaglic:Kacaniova:Knoechel:Lourenco:Mergulhao:Meyer:Nychas:Simoes:Tresse:Sternberg:2017,Janknecht:Melo:2003,Wolf:Crespo:Reis:2002,Surman:Walker:Goddard:Morton:Keevil:Weaver:Skinner:Hanson:Caldwell:Kurtz:1996}.

The imaging of biofilm metabolism has been reviewed in~\cite{Kuehl:Polerecky:2008}. Since biofilms grow fairly slowly~\cite[section~3]{Xi:Marks:Schlachter:Luo:Boppart:2006}, in the time span of hours and days, measurements of their density can be effectively performed continuously in time.
Thus, we feel justified in assuming that the biofilm density $M$ is accessible to measurement, as employed in our model.

Many spectral methods, e.g.  hyperspectral imaging~\cite{Kuehl:Polerecky:2008}, nuclear magnetic resonance imaging (NMR)~\cite{Grivet:Delort:2009} and spectroscopic techniques~\cite{Sankaran:Karampatzakis:Rice:Wohland:2018}, can distinguish concentrations of chemicals as a function of depth of certain chemicals biofilms feed on~\cite{Atci:Babauta:Ha:Beyenal:2017}, which also allows determining the substrate density~$S$.

\subsection{Models of biofilm growth}
There are several different models of biofilm growth; we refer to the reviews~\cite{Mattei:Frunzo:DAcunto:Pechaud:Pirozzi:Esposito:2018,Horn:Lackner:2014,Klapper:Dockery:2010,Wang:Zhang:2010}.
The model we are using makes the following noteworthy simplifications or assumptions:
\begin{itemize}
\item Substrate is governed by a diffusion equation. An equation of fluid flow would be more realistic in many situations.
\item The growth rate of the biofilm is proportional to the density of substrate for small substrate densities. The proportionality might not be true in biofilms, though it holds for free-floating bacteria~\cite{Moeller:Kristensen:Poulsen:Carstensen:Molin:1995}. A sufficiently small half-saturation constant $K_4$ makes the proportionality only true for very small values of $S$.
\item The diffusion rate of the substrate is a function of biofilm density. Our model ignores this for the sake of simplicity. However, the fact that the biofilm consumes the nutrients will effectively slow down the diffusion of nutrients deep into regions of biofilm with high density.
\end{itemize}
The book of Efendiev~\cite{Efendiev:2013} discusses the model \eqref{eq:biofilm} that we use here in detail and presents several generalizations, with numerics studied in~\cite{Ghasemi:Eberl:2018}.

\subsection{Related inverse problems }
Prior research on inverse problems of cell biology and population dynamics concerns, for example, population models~\cite{Duelk:2015}, biochemical reactions~\cite{Duelk:2015}, antibiotics~\cite{Serovajsky:Nurseitov:Kabanikhin:Azimov:Ilin:Islamov:2018}, and chemotaxis~\cite{Egger:Pietschmann:Schlottbom:2017}.

To our knowledge, there are no prior publications on inverse problems for biofilm growth.
There is an inverse problems study on diffusion through biofilms~\cite{Ma:Liu:Jiang:Liu:Tang:Ye:Zeng:Huang:2010}, but it does not discuss the growth rate of biofilms themselves.
Ad hoc estimation of some parameters has also been done in relation to antibiotics and biofilms in, for example,~\cite{Birnir:Carpio:Cebrian:2018}.

For a review of inverse problems for parabolic equations we refer to~\cite[section~9]{Isakov:2017}.
The literature on parameter recovery for quasilinear parabolic equations is extensive, especially when it comes to recovering a dependency on the solution itself; see the references in~\cite{Mierzwiczak:Kolodziej:2011} for many older papers.
For newer results, we mention the study in multiple dimensions by Egger, Pietschmann and Schlottbom~\cite{Egger:Pietschmann:Schlottbom:2015}.
Recovering multiple coefficients has been discussed in~\cite{Pilant:Rundell:1989,Lesnic:2002}.
There are more recent results on the recovery of diffusion coefficient in the presence of lower-order terms~\cite{Egger:Pietschmann:Schlottbom:2017} and on the simultaneous recovery of 
diffusion and lower-order terms~\cite{Egger:Pietschmann:Schlottbom:2014}.
The investigated problems are not degenerate.

Cortazar and Elgueta investigated a degenerate quasilinear problem~\cite{Cortazar:Elgueta:1990}.
Degenerate-singular problems have been studied less; we are aware of the works on the elliptic quasilinear variable exponent $p(\cdot)$-Laplace equation, where a linear factor in the diffusion is recovered~\cite{Brander:Winterrose:2019,Brander:Ringholm} and 
where the non-linearity itself is recovered~\cite{Brander:Siltakoski:2021}.
There is prior work on parameter recovery in non-linear systems of semilinear~\cite{Isakov:2001} and quasilinear nature~\cite{Egger:Pietschmann:Schlottbom:2017}.
We also mention inverse problems for non-linear reaction-diffusion equations with zeroth-order non-linearity~\cite{Duelk:2015,Sgura:Lawless:Bozzini:2018} and recovering lower-order terms in degenerate parabolic equations with a non-linearity in the zeroth-order terms~\cite{Tort:Vancostenoble:2012}. 

In this paper, we consider the singular-degenerate quasilinear parabolic system \eqref{eq:biofilm} and investigate the determination of all the parameters, both leading and lower order.
The task is manageable due to the very particular forward model, which means that we are only trying to recover a finite set of constant parameters.
Whenever needed, we also may assume interior knowledge of the system.

\section{Uniqueness results}
\label{sec:unique}
In this section we present several lemmas which indicate that a solution to the inverse problem is unique in an ideal situation with perfect measurements and no noise.
The linear terms are easy to recover, while for the non-linear terms we reduce the problem to the analysis of a one-dimensional real function, in the same spirit as in certain boundary determination 
procedures for $p$-Laplace type equations~\cite{Brander:2016:jan,Brander:Ringholm}.

We note that the required derivatives exist in the classical sense due to the regularity of the solutions, which is documented in Lemma~\ref{lemma:regularity} of Appendix \ref{AppendixA} for $F = G \equiv 0$.

\subsection{Equation for the substrate}
The results in this section follow immediately from the equation for the substrate in \eqref{eq:biofilm}, namely,
\begin{equation}
\label{eq:only_substrate}
\doo_t S = d_1 \Delta_x S - K_1 \frac{SM}{K_4 + S} +F, \quad (x,t) \in \Omega \times \R_{+}.
\end{equation}

\begin{lemma}
\label{lemma:d1}
If assumption~\ref{ass:iii} is satisfied,
\begin{eqnarray}
d_{1} = \frac{\doo_{t} S(x_{0},t_{0}) - F(x_{0},t_{0})}{\Delta_x S(x_0,t_0)}. \label{eq.6}
\end{eqnarray}
\end{lemma}
The lemma shows that the recovery of $d_1$ is unique. Moreover, $d_{1} \in (0,\infty)$ if the right-hand side of \eqref{eq.6} is a positive number. 
\begin{proof}
When evaluated at $(x_{0},t_{0})$, the last term in equation~\eqref{eq:only_substrate} vanishes under the assumption ~\ref{ass:iii}, and the identity \eqref{eq.6} follows immediately.
\end{proof}

The equation~\eqref{eq:only_substrate} for the substrate can be re-arranged to read
\begin{eqnarray}
S(x,t)M(x,t) K_1 + \parens{\doo_t S(x,t) - d_1 \Delta_x S(x,t)} K_{4} = -\parens{\doo_t S(x,t) - d_1 \Delta_x S(x,t) -F(x,t)} S(x,t). \label{eq.7}
\end{eqnarray}

\begin{lemma}
\label{lemma:linsub}
Suppose $d_1$ is known and assumption~\ref{ass:iv} is satisfied. Then, the recovery of $K_1$ and $K_4$ is unique.
\end{lemma}
\begin{proof}
Denote
\[
S_{i} = S(x_i,t_i), \quad M_i = M(x_i,t_i),
\quad c_i = \doo_t S(x_i,t_i) - d_1 \Delta_x S(x_i,t_i) -F(x_{i},t_{i}),
\quad i \in \{1,2\}.
\]
Then, on applying equation~\eqref{eq.7} at the points $(x_i,t_i)$ we obtain
\begin{eqnarray}
K_4 c_i + K_1 S_iM_i = - S_i c_i, \quad i = 1,2. \label{eq.8}
\end{eqnarray}
Equations \eqref{eq.8} form a linear system of two equations with two unknowns $K_{1}$ and $K_{4}$, which has a unique solution if and only if the determinant $(c_{1}S_{2}M_{2} - c_{2}S_{1}M_{1})$ of the 
system is non-zero, which is equivalent to the assumption~\ref{ass:iv} being satisfied. Then, the solution of \eqref{eq.8} is given by 
\begin{eqnarray}
K_{1} = \frac{(S_{1}-S_{2}) c_{1} c_{2}}{c_{1}S_{2}M_{2} - c_{2}S_{1}M_{1}}, \quad K_{4} = \frac{S_{1}S_{2}(c_{2} M_{1} - c_{1} M_{2})}{c_{1}S_{2}M_{2} - c_{2}S_{1}M_{1}}. \label{eq.9}
\end{eqnarray}
Moreover, $K_{1}$ and $K_{4} \in (0,\infty)$ if the right hand sides of the expressions in \eqref{eq.9} are positive numbers.
\end{proof}


\subsection{Linear factors in the equation for biofilm}

We proceed by integrating over the domain and using both the divergence theorem and boundary conditions on the biofilm density $M$.
To that end, we recall that the previous results imply in particular that $K_4$ can be recovered uniquely
if we assume the existence of certain special points.

\begin{proposition}
\label{prop:integrate}
Let $K_{4}$ be known. Suppose $S$ and $M$ are known in the entire spatial domain~$\Omega$ at the times $t_3 \neq t_4$, where assumption~\ref{ass:v} is satisfied. 
Then, the recovery of $K_2$ and $K_3$ is unique.
\end{proposition}

\begin{proof}
By integrating the equation for the biofilm density in~\eqref{eq:biofilm} over space and using the divergence theorem, the highly non-linear term becomes
\begin{equation}
\int_{\doo \Omega} \frac{M^b}{\parens{1-M}^{a}} \nabla_x M \cdot \nu \; \der S(x), \nonumber
\end{equation}
where $\nu$ denotes the outward unit normal to the boundary $\partial \Omega$. This term vanishes due to the zero-Dirichlet boundary condition on $M$.
We are left with the equation
\begin{eqnarray}
\int_\Omega ( \doo_{t} M -G) \der x = - K_2 \int_\Omega M \der x + K_3 \int_\Omega \frac{SM}{K_4+S} \der x. \label{extraa}
\end{eqnarray}
This is a linear equation with two unknowns, and hence the linear independence~\eqref{eq:iii_li} in assumption~\ref{ass:v} implies the uniqueness in finding the 
coefficients $K_2$ and $K_3$. Moreover, $K_{2} \in [0,\infty)$ and $K_{3} \in (0,\infty)$ if the corresponding expressions giving $K_{2}$ and $K_{3}$ out of the system of equations obtained by applying \eqref{extraa} at $t=t_{3}$ and $t=t_{4}$, e.g. by Cramer's rule, possess these sign properties.
\end{proof}


\subsection{Critical points when substrate and some constants K are known}
\label{subsec:nonlinear_injectivity}



In this section, we assume that $S$ and the constants $K_2,K_3$ and $K_4$ are known, as well as the biofilm density~$M$. The employed methodology is similar in principle to the one in \cite{Lorz:Pietschmann:Schlottbom:2019}, and it is based on evaluating the governing equation for the biofilm's density at a finite set of discrete points in order to obtain a linear system that can be solved to uniquely determine the sought parameters.

Consider a point $(x_0,t_0)$ as in assumption \ref{ass:vi} such that $0 < M(x_0,t_0) < 1$, $\Delta_x M(x_0,t_0) \neq 0$ and $\nabla_x M(x_0,t_0) = 0$.
(According to the regularity theory $\nabla_x M$ makes sense pointwise in cylinders where $M$ is bounded away from zero and one.)
At this stage, we have
\begin{equation}
\doo_t M(x_0,t_0) -G(x_{0},t_{0}) =d_{2} \frac{M^b(x_0,t_0)}{\parens{1-M(x_0,t_0)}^a} \Delta_x M(x_0,t_0) - K_2 M(x_0,t_0) + K_3 \frac{S(x_0,t_0)M(x_0,t_0)}{K_4+S(x_0,t_0)}. \nonumber
\end{equation}
We rewrite this equation as
\begin{eqnarray}
N := \parens{\doo_t M(x_0,t_0)  -G(x_{0},t_{0}) + K_2 M(x_0,t_0) - K_3 \frac{S(x_0,t_0)M(x_0,t_0)}{K_4+S(x_0,t_0)}}\parens{\Delta_x M(x_0,t_0)}^{-1} \nonumber \\
=d_2 \frac{M^b(x_0,t_0)}{\parens{1-M(x_0,t_0)}^a}. \nonumber
\end{eqnarray}
Since~$N$ is known, we also know the mapping $M \mapsto d_2 \lambda(M) = d_2 \frac{M^b}{\parens{1-M}^a}$ at critical points with non-zero values of $M$ where $\Delta M$ does not vanish.
Hence, we also know
\begin{equation}
M \mapsto \log \parens{d_2 \lambda(M)} = \log (d_2) + b \log \parens{M} - a \log \parens{1-M} \label{equationn8}
\end{equation}
at the critical points.

Suppose there are three critical points with non-vanishing Laplacian, with the values of $M$ denoted by $M_1$, $M_2$ and $M_3$, which are all distinct from each other, as in assumption~\ref{ass:vi}.
Then, from \eqref{equationn8} we have three equations
\begin{eqnarray}
-\log \parens{1-M_i} a + \log \parens{M_i} b + \log \parens{d_2} = N_i, \quad i =1, 2, 3, \label{extra}
\end{eqnarray}
and we would like to show that the three vectors
\begin{equation}\label{eq:threevectors}
\parens{-\log \parens{1-M_i}, \log \parens{M_i}, 1}, \quad i =1, 2, 3,
\end{equation}
are linearly independent, which would allow solving uniquely for $a$, $b$ and $d_2$.
This is equivalent to saying that, 
if there exist real numbers $\zeta_1$, $\zeta_2$ and $\zeta_3$ such that
\begin{equation}
\label{eq:independence}
- \zeta_1 \log (1-M_i) + \zeta_2 \log (M_i) = -\zeta_3 \quad \mbox{for} \; i \in \{ 1,2,3 \},
\end{equation}
then $\zeta_1 = \zeta_2 = \zeta_3 = 0$. Let therefore~\eqref{eq:independence} be satisfied and
assume first that at least one of
the coefficients $\zeta_1$ and $\zeta_2$ is different from zero.
We consider the mapping
\[
M \mapsto \mathrm{F}(M) := -\zeta_1 \log(1-M) + \zeta_2 \log(M).
\]
Then
\[
\mathrm{F}'(M) = \frac{\zeta_1}{1-M} + \frac{\zeta_2}{M} = \frac{M (\zeta_1 - \zeta_2) + \zeta_2}{M(1-M)}.
\]
Since $\zeta_{1}$ and $\zeta_{2}$ are not simultaneously equal to zero, there exists at most one value $0 < M < 1$ such that $\mathrm{F}'(M) = 0$.
That is, the (continuous) function $\mathrm{F} \colon (0,1) \to \R$ has at most one critical point.
This in turn implies that the equation
$\mathrm{F}(M) = -\zeta_3$ has at most two different solutions,
which contradicts the assumption that we have three distinct
values $M_i$, $i \in \{1,2,3\}$, satisfying~\eqref{eq:independence}.
As a consequence, it follows that both $\zeta_1$ and $\zeta_2$ are necessarily
both equal to zero, which immediately implies that $\zeta_3 = 0$ as well.
Thus the three vectors in~\eqref{eq:threevectors} are linearly independent and therefore $a$, $b$, and $d_2$ are uniquely determined. Moreover, $a \in [0,\infty)$, 
$b \in [1,\infty)$ and $d_{2} \in (0,\infty)$ if the corresponding expressions giving them explicitly by solving the linear system of equations \eqref{extra} using, e.g.\ the Cramer's rule, possess these lower bound properties.

This proves theorem~\ref{thm:unique}, when combined with proposition~\ref{prop:integrate} and lemmas~\ref{lemma:d1} and \ref{lemma:linsub}.

\subsection{Example}
\label{subsec:ex11}
In this subsection, we give an example to show that the assumtions (i)-(iv) of theorem \ref{thm:unique} can be satisfied. We also point out the unique recovery of the vector of unknowns $\underline{X}$ as an outcome of equations \eqref{eq.6}, \eqref{eq.9}, \eqref{extraa} (applied at $t=t_{3}$ and $t=t_{4}$) and \eqref{extra}.

We take $\Omega = (0,1)$, $a=0$, $b=1$, $d_{1} = d_{2} = 1$, $K_{1} = K_{3} = K_{4} =1$, $K_{2}=0$, 
$$M^{\textrm{exact}}(x,t) = 4x(1-x) t e^{1-t}, \quad S^{\textrm{exact}}(x,t) = 1 - M^{\textrm{exact}}(x,t),$$
which satisfy the system \eqref{eq:biofilm} with $M_{0}=0$, $S_{0}=1$ with the source functions
\begin{eqnarray}
F(x,t) = 2 e^{1-t} \left\{ - 4t +2x(1-x)(t-1) + \frac{\left(1-4x(1-x)t e^{1-t} \right) x(1-x) t}{1- 2x(1-x) t e^{1-t}} \right\}, \nonumber \\
G(x,t) = 2 e^{1-t} \left\{ 2x(1-x)(1-t) + 8 t^{2} e^{1-t} (6x - 6 x^{2} -1) 
- \frac{\left(1-4x(1-x)t e^{1-t} \right) x(1-x) t}{1- 2x(1-x) t e^{1-t}} \right\}. \nonumber
\end{eqnarray}
One can easily check that the assumption (i) of theorem \ref{thm:unique} is satisfied for $x_{0}=1/2$ and $t_{0}=1$, and that equation \eqref{eq.6} of lemma \ref{lemma:d1} yields $d_{1} = 1$, as required. Also, the assumption (ii) of theorem \ref{thm:unique} is satisfied for $x_{1} = x_{2} =1/2$, $t_{1}=1/2$ and $t_{2}=1$, and equation \eqref{eq.9} of lemma \ref{lemma:linsub} yields $K_{1} = K_{4} = 1$, as required. Assumption (iii) of theorem \ref{thm:unique} is satisfied for $t_{3}=1/2$ and $t_{4}=1$, and equation \eqref{extraa} applied at $t=t_{3}$ and $t=t_{4}$ yields $K_{2} =0$ and $K_{3} = 1$, as required. Finally, assumption (iv) of theorem \ref{thm:unique} is satisfied for $x_{5}=x_{6}=x_{7}=1/2$, $t_{5}=1/3$, $t_{6}=1/2$ and $t_{7}=2/3$, and equations \eqref{extra} yield $a=0$, $b=1$ and $d_{2} =1$, as required.

\section{Numerical solution of the direct problem}
Henceforth, for the numerical investigation, we specialize on the one-dimensional case. Similar computations can be performed in higher dimensions. 

Consider the following non-linear parabolic system of equations:
\begin{eqnarray}\label{eq11}
\begin{cases}
\frac{\partial S}{\partial t}=d_1\frac{\partial^2S}{\partial x^2}-K_1\frac{SM}{K_4+S}+F(x,t),\quad (x,t)\in (0,1)\times(0,T),\\
\frac{\partial M}{\partial t}=d_2\frac{\partial}{\partial x}(\lambda(M)\frac{\partial M}{\partial x})-K_2M+K_3\frac{SM}{K_4+M}+G(x,t),\quad (x,t)\in (0,1)\times(0,T),\\
S(0,t)=\mu_1(t),\quad S(1,t)=\mu_2(t),\quad t\in(0,T)\\
M(0,t)=\mu_3(t),\quad M(1,t)=\mu_4(t),\quad t\in(0,T),\\
S(x,0)=S_0(x),\quad M(x,0)=M_0(x),\quad x\in[0,1],
\end{cases}
\end{eqnarray}
where $T>0$ is a final finite time of interest, $S(x,t)$, $M(x,t)$ are two unknown functions, $F(x,t)$ and $G(x,t)$ are known source terms, $(\mu_i(t))_{i=\overline{1,4}}$ are Dirichlet boundary data, $S_{0}(x)$ and $M_0(x)$ represent the initial status of 
the system, $\lambda(M)=\frac{M^b}{(1-M)^a}$, and $a\geq0$, $b\geq1$, $d_1>0$, $d_2>0$, $K_1\geq0$, $K_2\geq0$, $K_3\geq0$ and $K_4>0$ are constants.


The numerical solution for solving the direct problem \eqref{eq11} is described in Appendix \ref{AppendixB}.

\subsection{Example 1}
\label{subsec:ex1}
We take $T=1$, $K_{1} = K_{2} = K_{3} = K_{4} = d_{1} = d_{2} = 1$, $\mu_{1}(t)=\mu_{2}(t)=1$, $\mu_{3}(t)=\mu_{4}(t)=0$, and
\begin{eqnarray}
F(x,t)= x-x^2+2(t+1)+\frac{[1+(x-x^2)(t+1)](x-x^2)e^{-t}}{2+(t+1)(x-x^2)},
\nonumber 
\end{eqnarray}
\begin{eqnarray}
G(x,t)= -\frac{(1-2x)^2(x-x^2)^2e^{-4t}}{[1-(x-x^2)e^{-t}]^2}-\frac{2(x-x^2)(1-5x+5x^2)e^{-3t}}{1-(x-x^2)e^{-t}} \nonumber \\
-\frac{[1+(x-x^2)(t+1)](x-x^2)e^{-t}}{2+(x-x^2)(t+1)}, \label{eqG}
\end{eqnarray}
\begin{eqnarray}
S_{0}(x)= 1+x-x^2,\quad M_0(x)=x-x^2, \label{eqS0}
\end{eqnarray}
and 
\begin{eqnarray}
a=1,\quad b=2.
\label{eq31}
\end{eqnarray}

Then, with this input data, the analytical solution of the direct problem \eqref{eq11} is given by
\begin{eqnarray}
S^{\text{exact}}(x,t)=1+(x-x^2)(t+1),\quad M^{\text{exact}}(x,t)=(x-x^2)e^{-t}.
\label{eq32}
\end{eqnarray}

\noindent
The numerical results obtained by solving the direct problem \eqref{eq11} using the numerical finite-difference method (FDM) described in Appendix \ref{AppendixB} 
are presented in Table 1 and excellent agreement with the analytical solution \eqref{eq32} can be observed. Furthermore, the $l^{\infty}$-errors decrease as the mesh size is refined and the second-order convergence estimate $\textrm{max} \left\{ \|S-S^{\mathrm{exact}}\|_\infty,\|M-M^{\mathrm{exact}}\|_\infty \right\} \leq A \left( (\Delta x)^{2} + (\Delta t)^{2} \right)$ with $A=0.3$ is consistent with that indicated in \cite{Lees:1966} for scalar quasilinear parabolic equations. 

\begin{table}[H]
	\centering
	\caption{The $l^\infty$-errors $\|S-S^{\mathrm{exact}}\|_\infty$ and $\|M-M^{\mathrm{exact}}\|_\infty$ for various mesh sizes when solving the direct problem 
\eqref{eq11}.}
	\begin{tabular}{c|c|c|c}
		\hline
		$\Delta x$ & $\Delta t$ & $\|S-S^{\mathrm{exact}}\|_\infty$ & $\|M-M^{\mathrm{exact}}\|_\infty$ \\
		\hline
		0.1 & 0.1 & $1.74\times10^{-4}$ & $2.30\times10^{-3}$   \\
		\hline
		0.05 & 0.05 & $4.56\times10^{-5}$ & $8.02\times10^{-4}$   \\
		\hline
		0.01 & 0.01 & $1.87\times10^{-6}$ & $4.52\times10^{-5}$   \\
		\hline
	\end{tabular}
\end{table}

\section{Inverse problem}
In the previous Section 2, sufficient additional measurements of the densities of biofilm and substrate have been given that provide uniqueness in recovering  
the biological constants associated with \eqref{eq:biofilm}. However, although these extra measurements enable retrieving a unique solution it would be difficult to investigate their optimality in 
terms of their necessity. In addition, the conditions that are imposed in Theorem \ref{thm:unique} are difficult to satisfy numerically, especially when the measured data is contaminated with noise. Therefore, in this section we explore and investigate the numerical reconstruction of the sought biological constants from other types of physical measurements such as the time history of the 
boundary heat flux and the biomass/bioenergy of the biofilm density. Numerical evidence that is presented and discussed below indicates that it is possible to retrieve uniquely the eight-parameter vector $\underline{X}$ from such data.

\subsection{Determining $a$ and $b$}
\label{subsec:simple_inverse}
Let us consider first the case when the constants $(d_1,d_2,K_1,K_2,K_3,K_4) \in (0,\infty)^3 \times[0,\infty) \times(0,\infty)^{2}$ are known and we wish to determine the pair of 
powers $(a,b)\in[0,\infty)\times[1,\infty)$ expressing the nonlinearity of the diffusivity~$\lambda(M)=M^b/(1-M)^a$. 

Consider the input data as in Example~1, section~\ref{subsec:ex1}, but now $a$ and $b$ are unknown and have to be retrieved. In doing so, we measure the flux of $S$ at $x=0$ given by
\begin{eqnarray}
\label{eq:flux_measurement}
-d_{1} \partial_{x}S(0,t)=:q_0(t)=-t-1,\quad t\in[0,1], 
\label{eq33}
\end{eqnarray}
and minimize the functional $H:[0,\infty)\times[1,\infty)\rightarrow\mathbb{R}_+$ given by
\begin{eqnarray}
H(a,b):=\|q_0^c(t;a,b)-q_0(t)\|_{L^2(0,1)}^2,
\label{eq34}
\end{eqnarray}
where $q_0^c$ is the computed flux $-d_1\partial_xS(0,t)$ for given values of $a$ and $b$, and $q_0(t)=-t-1$ given in \eqref{eq33} represents the exact data corresponding to the exact values \eqref{eq31} of $a$ and $b$. Because the functional \eqref{eq34} depends on two variables only, the simplest way to minimize it is to plot it for many uniformly distributed values of $(a,b)$ within some sufficiently wide searching interval (assumed available information from the physics of the problem), say $[0,4] \times [1,4]$. 

The results for $H(a,b)$, $H(a,2)$ and $H(1,b)$ are displayed in Figures 1 and 2. Moreover, the minimal values and minimizers of these functionals are given in Table 2. From these figures and table it can be seen that if the FDM mesh size is sufficiently fine (for our 
example $\Delta x = \Delta t =0.01$) then, the global minimum of the functional $H(a,b)$ (and more clearly of $H(a,2)$ and $H(1,b)$) is attained at the true values \eqref{eq31}. From now on, we fix the mesh size $\Delta x = \Delta t =0.01$ in our remaining computations. 

\begin{table}[H]
	\centering
\caption{The minimum values and minimizers for various mesh sizes.}
	\begin{tabular}{|c|c|c|c|}
		\hline
$\Delta x = \Delta t$& 0.1 & 0.05 & 0.01 \\
		\hline
		$\min H(a,2)$ & 1.4E-6 & 1.6E-7 & 1.2E-9  \\
		\hline
		$a_{\min}$ & 1.75 & 1.5 & 1   \\
		\hline
		$\min H(1,b)$ & 1.4E-6 & 1.6E-7 & 1.2E-9   \\
		\hline
		$b_{\min}$ & 2 & 2&2\\
		\hline
	\end{tabular}
\end{table}

\begin{figure}[H]
	\centering
	\includegraphics[width=100mm]{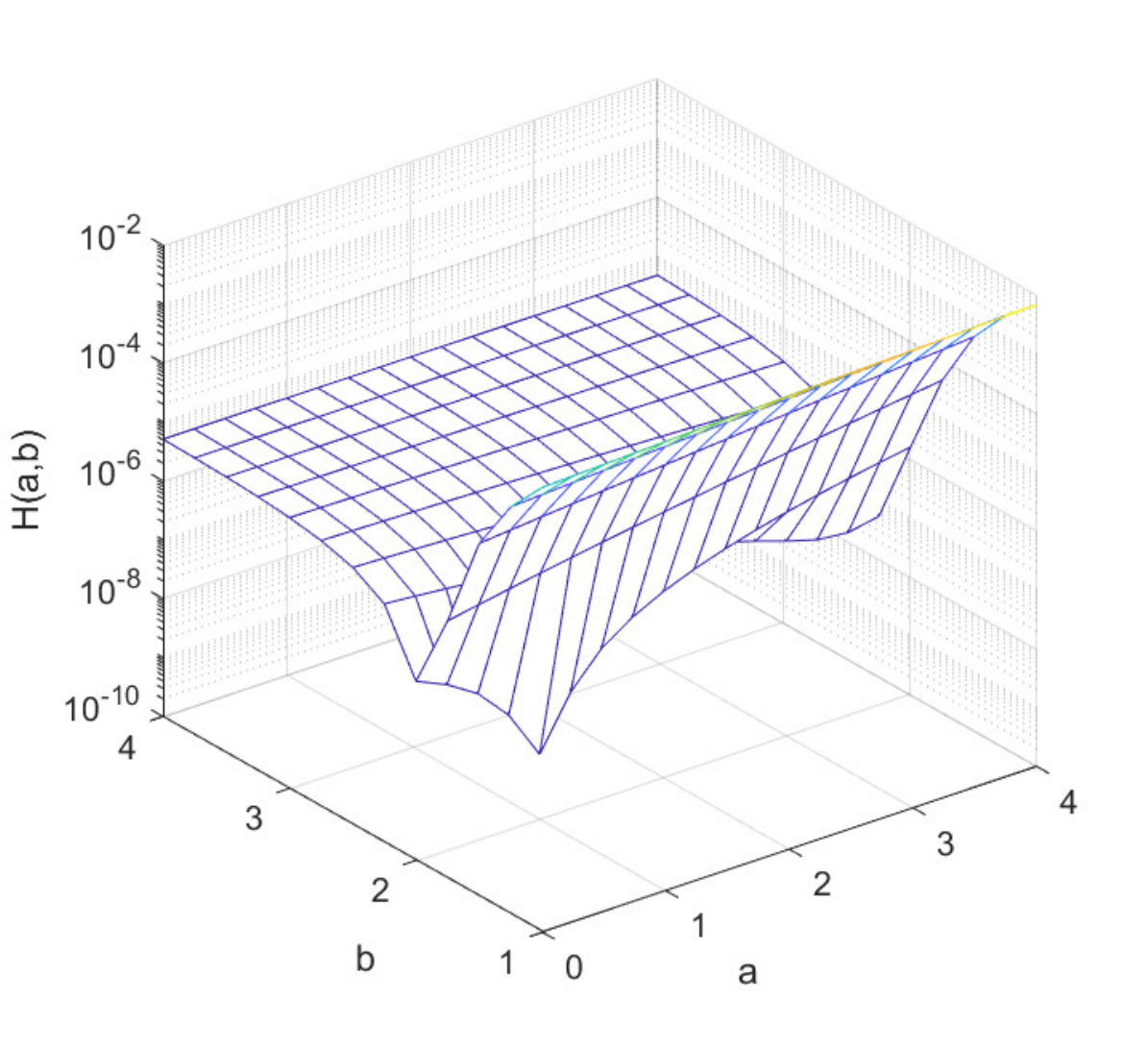}
	\caption{The cost functional $H(a,b)$ for $a\in[0,4]$, $b\in[1,4]$, for $\Delta x=\Delta t =0.01$.}\label{fig21}
\end{figure}

\begin{figure}[H]
	\centering
	\includegraphics[width=80mm]{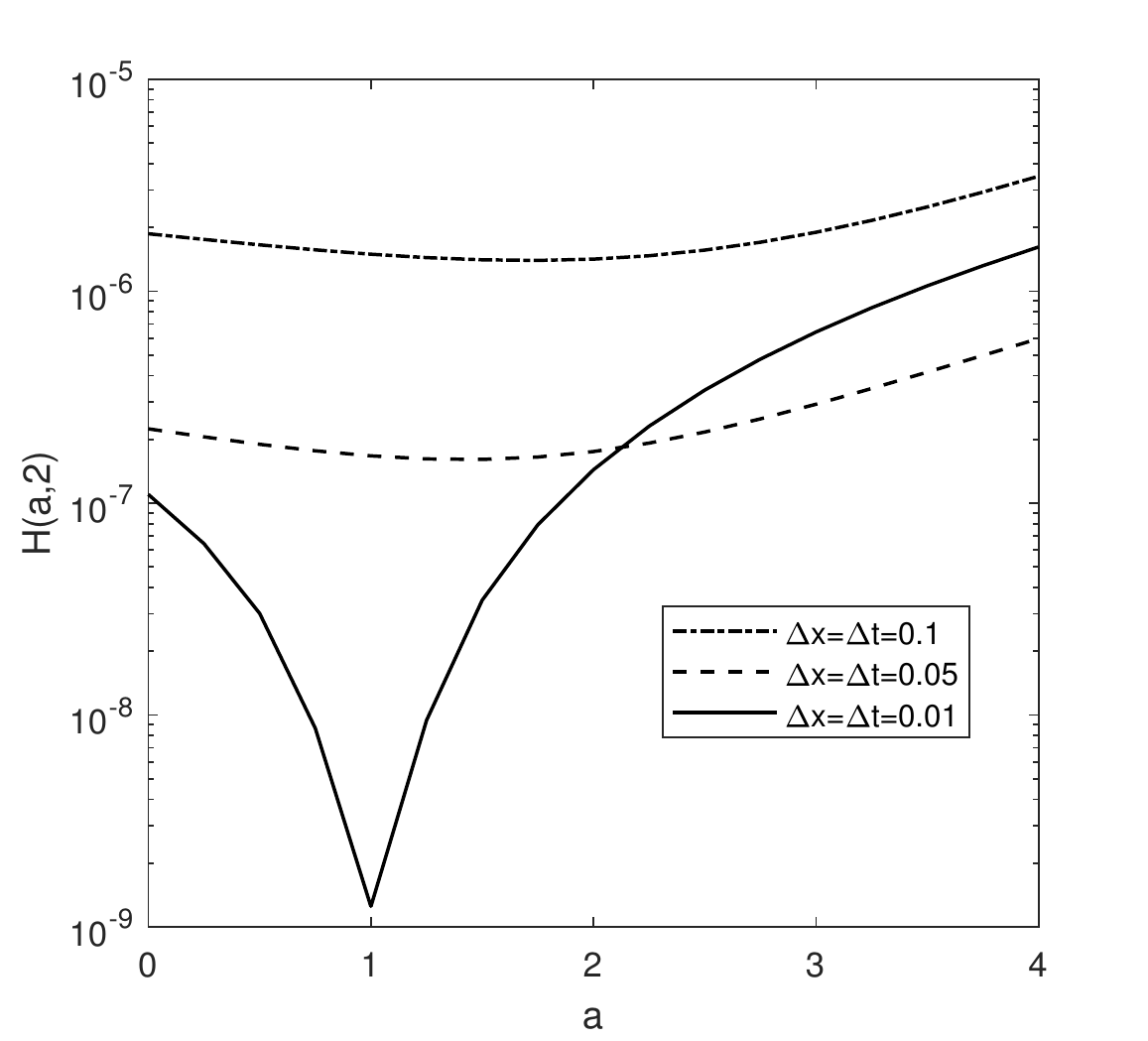}
		\includegraphics[width=80mm]{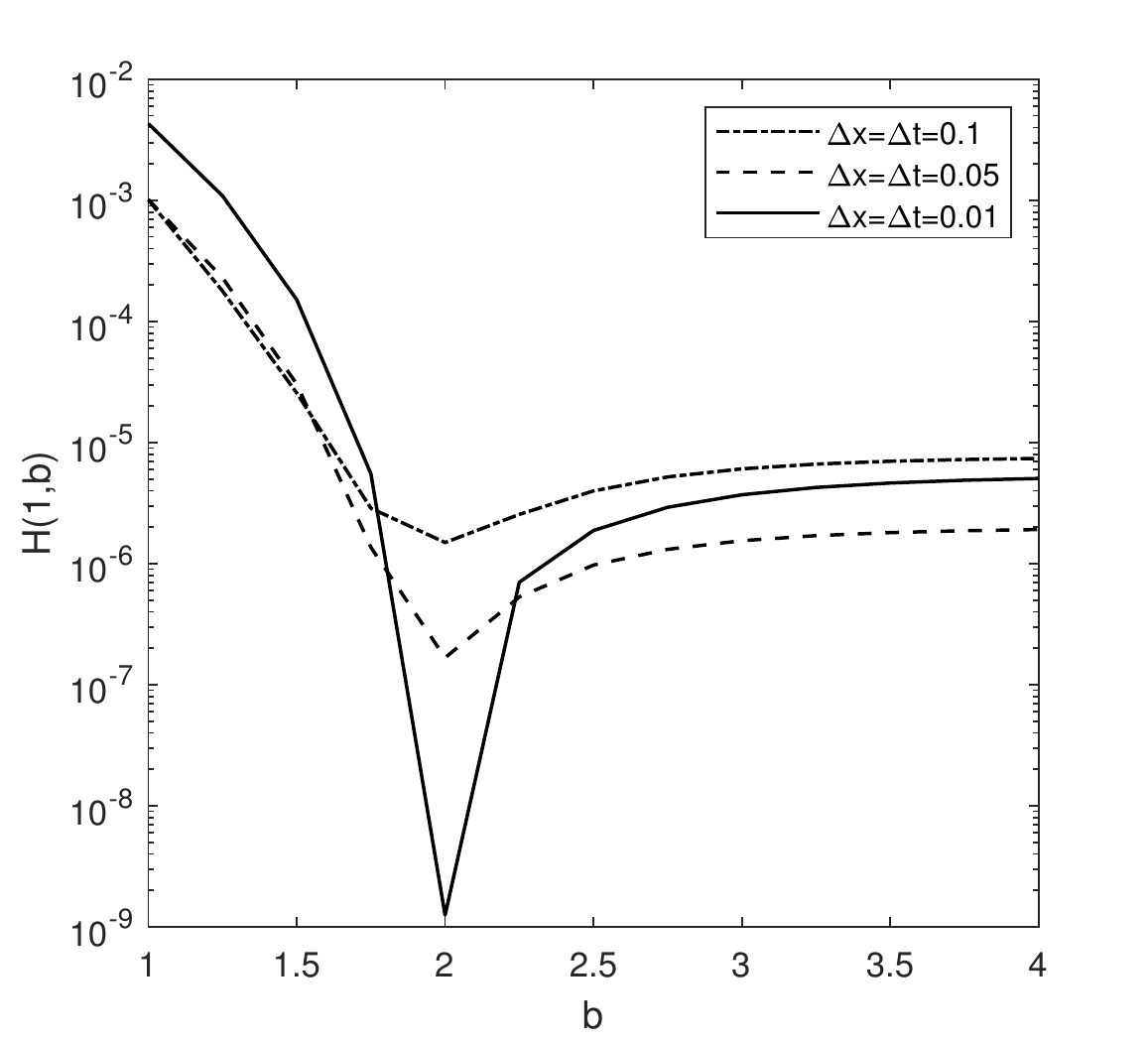}
	\caption{The cost functionals $H(a,2)$ and $H(1,b)$ for various mesh sizes.}\label{fig22}
\end{figure}

Of course, the above plotting technique becomes computationally inefficient and impractical if more than two parameters are to be estimated. In such a situation, gradient iterative minimization methods are desirable. We employ such an iterative method (\textit{lsqnonlin}) from the Matlab toolbox routines, with the simple bounds on the variables
$$ 0 \leq a \leq 4, \quad 1 \leq b \leq 4,$$
starting from various initial guesses: 
$$(a) \; a^{0}=0, \; b^{0}=1; \quad (b) \; a^{0}=2, \; b^{0}=1; \quad (c) \; a^{0}=3, \; b^{0}=1.$$ 
The obtained results are summarised in Table 3. These shows that a fast convergence towards the exact values \eqref{eq31} is achieved for either initial guesses (a)-(c). This independence on the initial guess indicates robustness of 
the iterative minimization.

\begin{table}[H]
	\centering
\caption{The minimal values and minimizers of $H(a,b)$ obtained using the \textit{lsqnonlin} routine for various initial guesses.}
	\begin{tabular}{|c|c|c|c|}
		\hline
      & \textrm{Guess (a)} & \textrm{Guess (b)} & \textrm{Guess (c)} \\
		\hline
		$\min H(a,b)$ &1.43E-11&2.00E-11& 1.19E-11   \\
		\hline
			$(a_{min}, b_{\min})$ &(0.9536, 1.9961)&(1.0552, 2.0046)&(1.0425, 2.0035)\\
			\hline
			$\textrm{No.\ of iterations}$ &15&8&10 \\
		\hline
	\end{tabular}
\end{table}

The analysis performed in this section introduced two different methods of reconstructing the power nonlinearities $a$ and $b$, both of them producing the same recovery of the desired unknowns. In the next section, we investigate whether it is possible to recover more than the two constants $a$ and $b$ from the flux measurement \eqref{eq33}.

\subsection{Determining $d_2$, $(K_{i})_{i=\overline{1,4}}$, $a$,  and $b$}
\label{subsec:intermediate_inverse}
Since $d_1$ appears explicitly in \eqref{eq34}, we assume that $d_{1} = 1$ is known, and try to recover $d_2$ and the four constants $(K_{i})_{i=\overline{1,4}} = \mathbf{1}$, in addition to the power nonlinearities $a=1$ and $b=2$.
As in section~\ref{subsec:simple_inverse}, we minimize the extended objective functional 
$H \colon  (0,\infty)^{2} \times [0, \infty) \times (0,\infty)^{2} \times [0, \infty) \times [1, \infty) \to \R_+$
given by
\begin{equation}
\label{eq:object_function}
H(d_2, (K_{i})_{i=\overline{1,4}}, a, b) = \norm{q_0^c \parens{t; d_2, (K_{i})_{i=\overline{1,4}}, a, b} - q_0(t) }^2_{L^2(0, 1)}
\end{equation}
subject to the simple bounds on the variables
\begin{eqnarray}
0 \leq a \leq 10^{10}, \quad 1 \leq b \leq 10^{10},  \quad 10^{-10} \le d_2 \le 10^{10}, \nonumber \\
10^{-10} \leq K_{1} \leq 10^{10}, \quad 0 \leq K_{2} \leq 10^{10}, \quad 10^{-10} \leq K_{3} \leq 10^{10}, \quad 10^{-10} \leq K_{4} \leq 10^{10}.
\label{eq:simple_bounds}
\end{eqnarray} 
This is accomplished using the iterative \textit{lsqnonlin} from the initial guess
\begin{eqnarray}
a^{0}=2, \quad b^{0}=1, \quad K_{i}^{0}= d_2^0 = 0.5 \; \textrm{for} \; i=\overline{1,4}. 
\label{eq:initial_guess}
\end{eqnarray}

\begin{figure}[H]
	\centering
	\includegraphics[width=100mm]{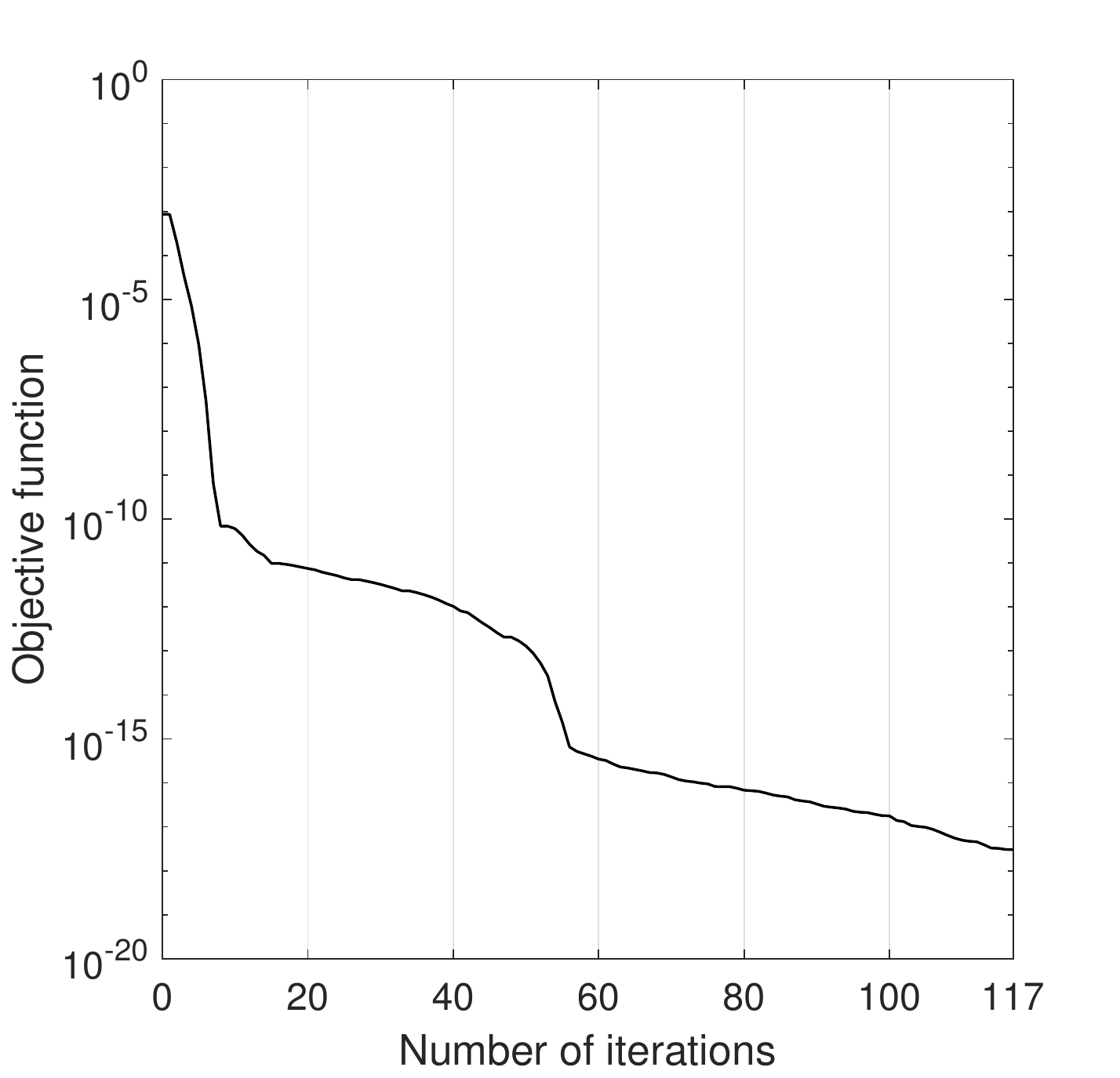}
	\caption{Convergence of the objective function~\eqref{eq:object_function}.}
	\label{fig:jeobj}
\end{figure}

The convergence of the objective function~\eqref{eq:object_function} with the number of iterations is illustrated in Figure~\ref{fig:jeobj}. The reconstructed values after 117 iterations of the \textit{lsqnonlin} routine are:
\begin{eqnarray}
a= 0.9999, \quad b = 2.0000, \quad d_2 =1.0001, \nonumber \\
K_{1}= 1.0003, \quad K_{2}= 0.9612, \quad K_{3}= 0.9320, \quad K_{4}= 1.0000. \nonumber
\end{eqnarray}

These numerical values are in good agreement with their true values, except for $K_2$ and $K_3$, where the errors around 5\%--10\% are larger.

\subsection{Determining all the eight unknowns $d_1, d_2, (K_i)_{i=\overline{1, 4}}, a$ and $b$}
In this section, we present the results of numerically retrieving all the eight unknowns $\underline X = \parens{d_1, d_2, (K_i)_{i=\overline{1, 4}}, a, b} \in \frak X$.
For $d_1$ we take the simple bounds $10^{-10} \le d_1 \le 10^{10}$, whilst for the remaining unknowns we take the bounds given in \eqref{eq:simple_bounds}.
We use the initial guess~\eqref{eq:initial_guess} together with two initial guesses for $d_1$, namely $d_1^0 = 0.5$ or $d_1^0 = 1.3$.
The results obtained by  minimizing the functional
$H \colon \frak X \to \R_+$ given by
\begin{equation}
\label{eq:object_function_hard}
H(\underline X) = \norm{q_0^c \parens{t; \underline X} - q_0(t) }^2_{L^2(0, 1)}
\end{equation}
are illustrated in Figure~\ref{fig:jeobj1} and Table~\ref{table:inverse_results}. 

\begin{figure}
	\centering
	\includegraphics[width=100mm]{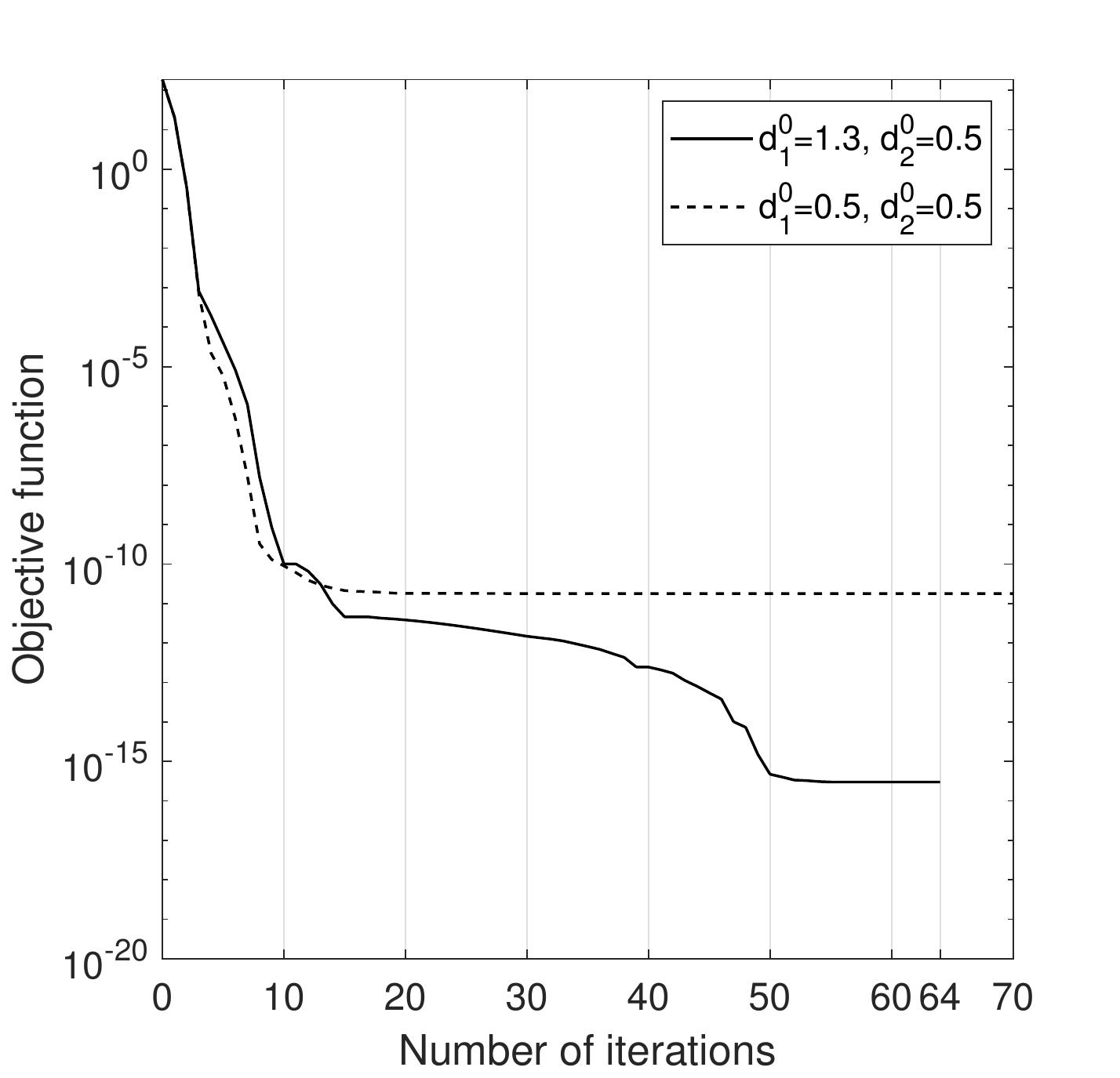}
\caption{Convergence of the objective function~\eqref{eq:object_function_hard} for the initial guess~\eqref{eq:initial_guess} and $d_1^0 \in \joukko{ 0.5 ,1.3}$.}
	\label{fig:jeobj1}
\end{figure}

\begin{table}[H]
	\centering
\caption{The minimizers of \eqref{eq:object_function_hard} obtained using the \textit{lsqnonlin} routine for the initial guess~\eqref{eq:initial_guess} and $d_1^0 \in \joukko{ 0.5 ,1.3}$.}
	\begin{tabular}{|c|c|c|c|}
		\hline
		Parameter 	& $d_1^0 = 0.5$	& $d_1^0 = 1.3$	& Exact	\\
		\hline
		$a$			& 2.1230				& 0.9973				& 1		\\
		\hline
		$b$			& 1.8140				& 2.0007				& 2		\\
		\hline
		$K_1$		& 0.7143				& 1.0032				& 1		\\
		\hline
		$K_2$		& 1.2392				& 0.6079				& 1		\\
		\hline
		$K_3$		& 0.3109				& 0.3127				& 1		\\
		\hline
		$K_4$		& 1.0001				& 1.0000				& 1		\\
		\hline
		$d_1$		& 1.0000				& 1.0000				& 1		\\
		\hline
		$d_2$		& 0.5750				& 1.0018				& 1		\\
		\hline
	\end{tabular}
\label{table:inverse_results}
\end{table}

The initial guess~\eqref{eq:initial_guess} with $d_1^0 = 0.5$ produces a locally convergent solution that got stuck in a local minimum of the functional~\eqref{eq:object_function_hard}.
The initial guess with $d_1^0 = 1.3$ yields a much lower objective function value with estimates closer to their true values, though with errors of similar kind in the biofilm growth and decay rates. In particular, the parameters $K_2$ and $K_3$ seem to be the most difficult to retrieve.
This is somewhat expected, since the constants~$K_2$ and $K_3$ do not appear in the first equation in~\eqref{eq:biofilm} for $S$ and in the inverse problem we minimize the flux $q_0(t)$ of the substrate~$S$.
Furthermore, we have also calculated the sensitivity coefficients (not illustrated) given by the partial derivatives of the measurement~$q_0$ with respect to the unknowns~$K_2$ and $K_3$, obtaining that these are correlated and small, of order $O(10^{-4})$.
This indicates that the parameters $K_2$ and $K_3$ are quite insensitive to the flux measurement~\eqref{eq:flux_measurement}, hence justifying their poorer retrievals illustrated in Table~\ref{table:inverse_results}.

Overall, the numerical investigation performed in this section so far indicates some reasonable retrievals of the coefficients, but it also highlights some difficulties encountered by the employed gradient iterative minimization.
In addition, there might be non-uniqueness issues related to the use of only the flux $q_0(t)$ as the measurement data, which will amplify the errors even further if noise was to be considered in~\eqref{eq33}.
Therefore, we consider extra measurement data given by the biomass/bioenergy of the biofilm density,
\begin{equation}
\label{eq:M_mass}
E_{M}(t) = \int_\Omega M(x, t) \der x = e^{-t}/6, \quad t \in [0, 1],
\end{equation}
calculated here for the analytical solution in section~\ref{subsec:ex1}.
Measurement of total mass (or energy) is customary in the modelling of diffusion processes~\cite{Cannon:Hoek:1986,Cannon:1963} and also very feasible in this case; see section~\ref{sec:imaging}.

The measurements \eqref{eq:flux_measurement} and \eqref{eq:M_mass} are imposed in the least squares sense by minimizing the objective functional $J \colon \frak X \to \R_+$ given by
\begin{equation} \label{eq:objective}
J \parens{\underline X} := \norm{q_0^c \parens{ t; \underline X}-q_0 (t)}^2_{L^2(0, 1)} + \norm{E^c_M (t; \underline X) - E_M (t)}^2_{L^2(0, 1)},
\end{equation}
where $E_M^c(t; \underline X)$ is the computed mass/energy using the trapezoidal rule and the homogeneous Dirichlet boundary conditions on $M$, as
\begin{eqnarray}
E_M^c (t; \underline X) = \Delta x \sum_{i=2}^{I-1} M \parens{ x_i, t; \underline X }. \nonumber
\end{eqnarray}
The convergence of the objective function~\eqref{eq:objective} and the numerically retrieved values for the eight unknowns in vector~$\underline X$ are given in Figure~\ref{fig:5} and Table~\ref{table:5}. Figure~\ref{fig:5} illustrates how the non-convex least-squares functional \eqref{eq:objective} can get stuck or not in a local minimum for certain initial guesses, (see the results for $d_{1}^{0} = 0.5$ compared to $d_{1}^{0} = 1.3$). 
Compared to Table~\ref{table:inverse_results}, the numerical results presented in Table~\ref{table:5} show much better reconstructions of the eight unknowns, 
including the estimates for $K_2$ and $K_3$ (though $K_3$ suffered with $d_1^0 = 0.5$), when the extra measurement~\eqref{eq:M_mass} is taken into account.

\begin{table}[H]
	\centering
\caption{The minimizers of \eqref{eq:objective} obtained using the \textit{lsqnonlin} routine for the initial guess~\eqref{eq:initial_guess} and $d_1^0 \in \joukko{ 0.5 , 1.3}$.}
	\begin{tabular}{|c|c|c|c|}
		\hline
		Parameter 	& $d_1^0 = 0.5$	& $d_1^0 = 1.3$	& Exact	\\
		\hline
		$a$			& 1.0041				& 1.0000				& 1		\\
		\hline
		$b$			& 2.0007				& 2.0000				& 2		\\
		\hline
		$K_1$		& 1.0032				& 1.0000				& 1		\\
		\hline
		$K_2$		& 0.5623				& 1.0005				& 1		\\
		\hline
		$K_3$		& 0.2268				& 1.0010		& 1		\\
		\hline
		$K_4$		& 1.0000					& 1.0000				& 1		\\
		\hline
		$d_1$		& 1.0000			& 1.0000				& 1		\\
		\hline
		$d_2$		& 1.0005				& 1.0000				& 1		\\
		\hline
	\end{tabular}
\label{table:5}
\end{table}

\begin{figure}
	\centering
	\includegraphics[width=100mm]{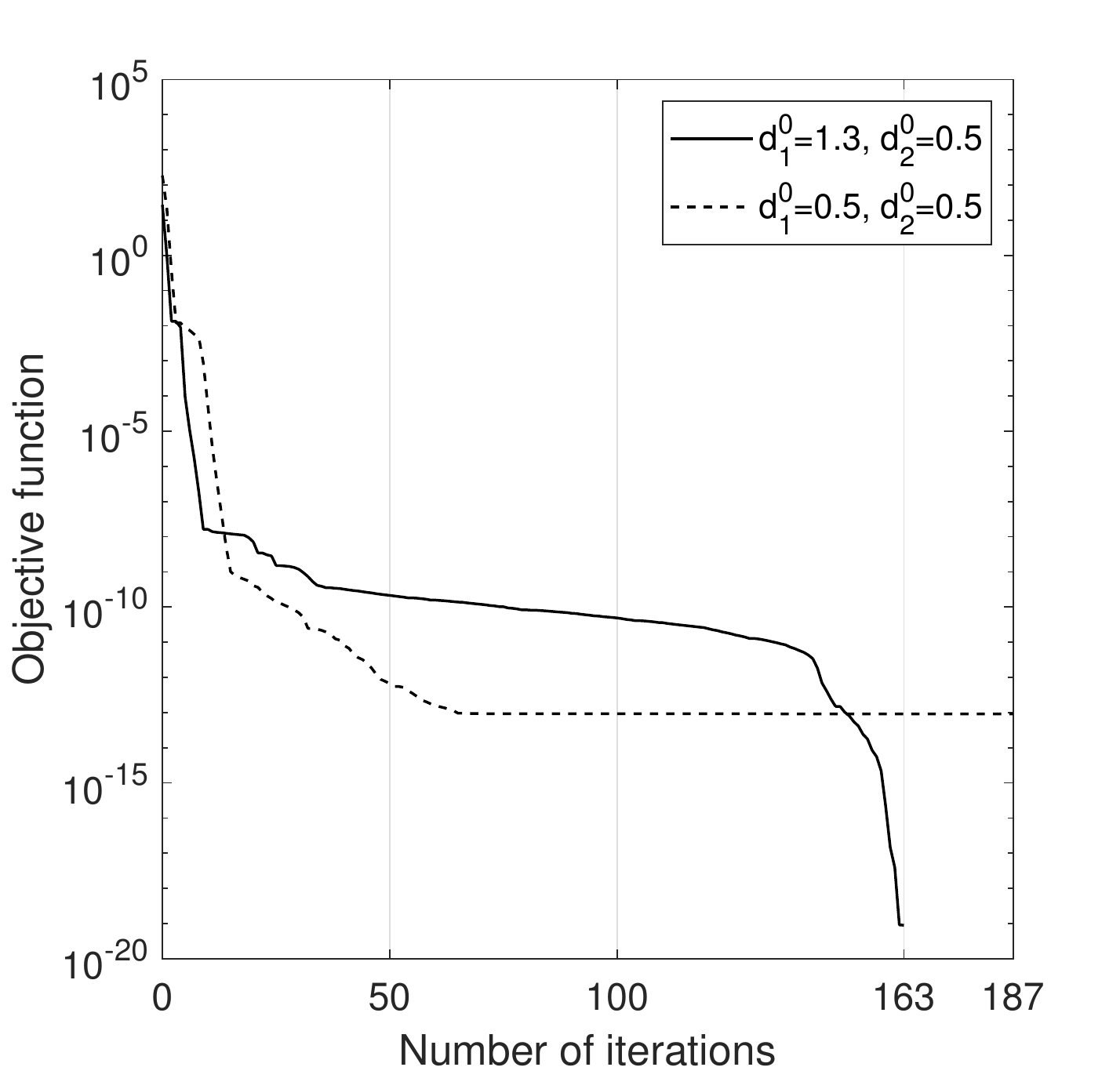}
	\caption{Convergence of the objective function~\eqref{eq:objective} for the initial guess~\eqref{eq:initial_guess}, and $d_1^0 \in \joukko{ 0.5 ,1.3}$.}
	\label{fig:5}
\end{figure}

For the initial guess~\eqref{eq:initial_guess} and $d_1^0 =1.3$, the reconstruction identifies exactly the true solution for the unknowns. In order to improve the robustness with respect to the other initial guess given 
by \eqref{eq:initial_guess} and $d_1^0 =0.5$, observe that upon integrating over space $\Omega = (0, 1)$ the second equation in \eqref{eq11} and using the homogeneous Dirichlet boundary condition $M|_{\doo \Omega \times \R_+} = 0$ of \eqref{eq:biofilm}, we obtain
\begin{equation}
\label{eq:44}
-K_2 E_M(t) + K_3 \int_{\Omega} \frac{S(x, t)M(x, t)}{K_4 + S(x, t)} \der x = E_M'(t) - \int_{\Omega} G(x, t) \der x, \quad t \in \R_+,
\end{equation}
which highlights a relationship between the unknowns $K_2$, $K_3$, and $K_4$.
In particular, letting $t \searrow 0$ in \eqref{eq:44}, we have
\begin{equation}
\label{eq:45}
-K_2 E_M(0) + K_3 \int_{\Omega} \frac{S_0(x)M_0(x)}{K_4 + S_0(x)} \der x = E_M'(0) - \int_{\Omega} G(x, 0) \der x.
\end{equation}
On substituting the expressions for $G(x,0)$, $S_{0}(x)$ and $M_{0}(x)$ from \eqref{eqG} and \eqref{eqS0} into \eqref{eq:45}, and evaluating the integrals involved using symbolic 
computations in MAPLE we obtain
\begin{eqnarray}
K_{2}=0.454822555 +K_{3} \left[ 1-6K_{4}+\frac{24K_{4}(K_{4}+1)}{\sqrt{5+4K_{4}}} \mbox{arctanh} \left( \frac{1}{\sqrt{5+4K_{4}}} \right) \right]. 
\label{K2}
\end{eqnarray}
This expression can be substituted into the problem in order to reduce the number of unknowns by 1, i.e. from 8 to 7 unknowns. Then, starting from the initial guess given by the second column of Table~\ref{table:5}, 
namely, 
\begin{eqnarray}
a^{0}=1.0041, \; b^{0}=2.0007, \; K_{1}^{0}= 1.0032, \; K_{3}^{0}=0.2268, \; K_{4}^{0}=d_{1}^{0}=1, \; d_{2}^{0}=1.0005,
\nonumber
\end{eqnarray}
we obtain in 10 iterations the true values (with 4 digits) of the 7 unknowns, plus $K_{2}$ via \eqref{K2}. 

Although no random noisy errors have been introduced in the flux $q_{0}(t)$ and the biomass $E_{M}(t)$, the analytical input data \eqref{eq:flux_measurement} and \eqref{eq:M_mass} already contain some numerical noise due to the fact 
that we are discretising the problem with a fixed mesh size and only in the limit $\Delta x = \Delta t \searrow 0$ this numerical noise disappears. Therefore, as a by-product, our numerical inversion has also been tested for stability with respect to 
this type of numerical noise in the input data. For general random noise in the input data 
\eqref{eq:flux_measurement} and \eqref{eq:M_mass}, the least-squares functional \eqref{eq:objective} may need to be penalised using  regularization. 


\section{Conclusions}
We have given an elementary injectivity proof that assumes information on the biofilm and substrate density, but otherwise has non-restrictive assumptions.
The method could work in other inverse problems with rich data, either within the domain or at a boundary point.

We have also illustrated numerically (in one spatial dimension with similar conclusions expected to hold also in higher dimensions) that recovering only a few parameters  is feasible from measurements of only the flux of the substrate, while as the number of unknown parameters increases, recovery becomes more difficult.
In particular, recovering the coefficients of biofilm growth, $K_2$ and $K_3$, seems challenging, but this can be overcome by further measuring the total biomass.  Of course, for inverting random noisy data or if the number of unknown parameter increases or if the parameters are variable with space and/or time, regularization 
may need to be employed~\cite{Chen:Jiang:2021}. In the near future, it is hoped that the mathematical model investigated in this paper using inverse problem techniques can be validated by inverting real raw data.

\subsection*{Acknowledgements}
We would like to Peter Lindqvist for discussions concerning the regularity of the solutions. The disussions with Markus Grasmair are also acknowledged. 
T.B.\ was partially funded by grant no.~4002-00123 from the Danish Council for Independent Research | Natural Sciences and the Research Council of Norway through the FRIPRO Toppforsk project `Waves and nonlinear phenomena'.

\bibliographystyle{plain}
\bibliography{math}

\appendix

\section{Classical regularity of the forward problem} \label{AppendixA}
In this appendix we briefly outline the classical regularity for the solution $(S,M)$ of the homogeneous problem associated to \eqref{eq:biofilm} obtained by taking $F=G=0$, namely, 
\begin{eqnarray} \label{eq:biofilm1}
\begin{cases}
\doo_t S = d_1 \Delta_x S - K_1 \frac{SM}{K_4 + S}, \quad (x,t) \in \Omega \times \R_{+}, \\
\doo_t M = d_2 \nabla_x \cdot \parens{\frac{M^b}{\parens{1-M}^a} \nabla_x M} - K_2 M + K_3 \frac{SM}{K_4+S}, \quad (x,t) \in \Omega \times \R_{+}, \\
S|_{\doo \Omega \times \R_+} = 1, \quad M|_{\doo \Omega \times \R_+} = 0, \\
S|_{t=0} = S_0, \quad M|_{t=0} = M_0, \end{cases} 
\end{eqnarray}
outside the set~$\doo \joukko{(x, t) \in \Omega \times \R_+; M(x, t) = 0}$.
Since the degeneracy in our equation is of porous medium type, we do not expect the solution to be smooth everywhere~\cite{Vazquez:2007}.
We have not considered the general equation with source terms, since it would require re-establishing the foundational results in~\cite{Efendiev:2013}, which we have to leave outside the scope of the present article.

We use the exponent $\alpha \in (0,1)$ as a generic Hölder exponent. The Hölder regularity in time of solution is only $\alpha/2$ if the spatial Hölder regularity is of order $\alpha$, but we have ignored this to simplify notation.
First, we establish the Hölder regularity of solution, after which we consider all the terms aside from the highest order terms to be stationary functions, and thereafter use the principle that 
the linear heat equation with Hölder terms has $C^{2+\alpha}$-regularity, and half of that in time.

Some results only hold in sets where $0 < \eps < M < 1 - \eps$. 
The upper bound is not a significant issue, since under the Dirichlet boundary conditions in equation~\eqref{eq:biofilm} the following result holds~\cite[theorem 5.1]{Efendiev:2013}: If $\norm{M_{0}}_{L^\infty(\Omega)} < 1$, then $\norm{M(x, t)}_{L^\infty(\Omega \times \R_+)} < 1$. The lower bound however remains an issue.

The lemmas below hold for as long the solution $(S, M)$ exists and remains bounded. This is not restricted to only the Dirichlet boundary conditions in~\eqref{eq:biofilm}, where the existence and boundedness is guaranteed for all time~\cite[section~5.1]{Efendiev:2013}, and these results hold more generally.
Since the model might be used with some other boundary conditions, too, we have decided to include this turn of phrase, rather than assume the precise Dirichlet data.

There exists a unique solution to the direct problem~\eqref{eq:biofilm1}~\cite[section~5.1]{Efendiev:2013}.
Since we know {\it a priori} from~\cite[chapter~5, theorem~5.3]{Efendiev:2013} that $S$ and $M$ are bounded, classical theory implies~\cite[chapter~V, section~1, theorem~1.1]{Ladyzenskaja:Solonnikov:Uralceva:1968}, \cite[chapter~2, section~1, remark~1.1.; see also chapter~3, section~1, theorems~1.1 and 1.2]{DiBenedetto:1993} that $S \in C^\alpha(\Omega \times \R_+)$ and, for any $\eps >0$, $M \in C^\alpha_{\text{loc}}\parens{ \joukko{(x, t) \in \Omega \times \R_+; 1- \eps > M(x,t) > \eps > 0} }$, where the regularity of $S$ is only local unless we make reasonable assumptions on boundary and initial values and the geometry of the domain $\Omega$.
Efendiev~\cite[chapter~5]{Efendiev:2013} also proves that $S \in H^1(\Omega)$ and $M \in H^s(\Omega)$ for $0 < s < 1/(b+1)$.

In \cite[chapter~V, section~3, theorem~3.1]{Ladyzenskaja:Solonnikov:Uralceva:1968} it is proven that $\nabla_x S \in C^\alpha_{\text{loc}}$ in cylinders and
$$
\int_{\text{cylinder}} \abs{\doo_t S}^2 \der x \der t < \infty, \quad \int_{\text{cylinder}}  \abs{ \doo_{x_i}\doo _{x_j} S}^2 \der x \der t < \infty.
$$
Similar results hold for $M$ in cylinders where it is bounded away from $0$ and $1$.
The results are global for $S$ if $\doo \Omega$ is regular enough and the boundary conditions likewise.
There are extra regularity assumptions that are made for the sake of simplicity, which can be removed by the methods in~\cite[section~IV]{Ladyzhenskaya:Ural'tseva:1968}.
Hence, we have $\doo_t S, \doo_{x_i} \doo_{x_j} S \in L^2_{\text{loc}}$ and likewise for $M$ in cylinders where $M$ is bounded away from zero and one.

\begin{lemma}
We have, 
as long as the solution $(S, M)$ of \eqref{eq:biofilm1} exists and remains bounded:
\begin{eqnarray}
\norm{S}_{L^\infty} \le 1, \quad \nabla_x S \in C^\alpha_{\text{loc}}\parens{\Omega \times \R_+}, \quad \doo_t S \in L^2_{\text{loc}}\parens{\Omega \times \R_+}, \quad 
D_x^2 S \in L^2_{\text{loc}}\parens{\Omega \times \R_+}. \nonumber
\end{eqnarray}
\end{lemma}

\begin{lemma}
As long as the solution $(S,M)$ of \eqref{eq:biofilm1} exists and remains bounded, we have that $\norm{M}_{L^\infty(\Omega \times \R_+)} \le 1$, and for all space-time cylinders $Q \Subset \Omega \times \R_+$ 
for which there exists $\eps > 0$ with
\begin{eqnarray}
0 < \eps < M(x,t) < 1-\eps, \quad (x,t) \in Q, \nonumber
\end{eqnarray}
we have that
\begin{eqnarray}
\nabla_x M \in C^\alpha\parens{Q}, \quad \doo_t M \in L^2\parens{Q}, \quad D_{x}^{2} M \in L^2\parens{Q}. \nonumber
\end{eqnarray}
\end{lemma}

Next, we reconsider the pair of equations in of \eqref{eq:biofilm1} as the linear parabolic equations
\begin{eqnarray}
\begin{cases}
\doo_t S &= d_1 \Delta_x S + f(x,t),\;\; (x,t) \in \Omega \times \R_+ \\
\doo_t M &= d_2 \Delta_x M + g(x,t),\; \; (x,t) \in \Omega \times \R_+,
\end{cases} \nonumber
\end{eqnarray}
where we have omitted the initial and boundary conditions and have written
\begin{align}
f(x,t) &= - K_1 \frac{S(x,t)M(x,t)}{K_4 + S(x,t)}, \nonumber \\
g(x,t) &= d_2 \frac{bM^{b-1}\parens{1-M}^a + a\parens{1-M}^{a-1}M^b}{\parens{1-M}^{2a}} \abs{\nabla_x M}^2 - K_2 M + K_3 \frac{SM}{K_4+S}. \nonumber
\end{align}
As long as we restrict ourselves to a smooth set where $0 < \eps < M < 1-\eps$, both $f$ and $g$ are Hölder continuous -- multiplication, division, and addition maintain local Hölder-continuity, 
 and since $M$ and $1-M$ are restricted away from zero, the possibly negative powers involving $a$ or $b$ also maintain Hölder-continuity, though the exponent~$\alpha$ may change.
In particular, the power functions $x \mapsto x^p$ (for any $p \in \R$) are smooth with all derivatives bounded when $x$ is restricted to a compact subset of positive real numbers.

Now, \cite[chapter~V, section~6, theorem~6.2]{Ladyzenskaja:Solonnikov:Uralceva:1968} gives that first time derivative and second spatial derivatives of both $S$ and $M$, in regions where $M$ is bounded away from zero, are Hölder-continuous.
Also, since $S$ solves a heat equation when $M=0$, it is regular in open sets where $M$ vanishes.

\begin{lemma}[Regularity lemma]
\label{lemma:regularity}
At any~$(x,t) \notin \doo \joukko{(x, t) \in \Omega \times \R_+ ; M(x, t) = 0}$ where $0 < M(x,t) < 1$, all derivatives $\doo_t M(x, t)$, $\doo_t S(x, t)$, $\Delta_x M(x, t)$ and $\Delta_x S(x, t)$ exist pointwise as classical derivatives.
\end{lemma}



\section{Numerical solution of the forward problem} \label{AppendixB}
The linearly implicit three-level finite-difference scheme \cite{Lees:1966} is applied to obtain the numerical solution to the non-linear parabolic direct (forward) problem \eqref{eq11}. For numerical discretization, a rectangular grid is constructed by subdividing the solution domain into $I\times N$ subintervals of the step lengths $\Delta x$ and $\Delta t$ in space $x$ and time $t$ directions, where $I$ and $N$ are two positive integers greater than 2.
Taking $\Delta x=\frac{1}{I-1}$ and $\Delta t=\frac{T}{N-1}$, then
$$x_i=(i-1)\Delta x,\quad i=\overline{1,I},\quad t_n=(n-1)\Delta t,\quad n=\overline{1,N}.$$
Denote $S_i^n:=S(x_i,t_n)$, $M_i^n:=M(x_i,t_n)$, $F_i^n:=F(x_i,t_n)$ and $G_i^n:=G(x_i,t_n)$ for $i=\overline{1,I}$ and $n=\overline{1,N}$. Then, for $i=\overline{2,I-1}$, we have
\begin{align*}
&S_i^1=S_0(x_i),\quad M_i^1=M_0(x_i),\\
&S_i^2=S_i^1+\frac{d_1\Delta t}{(\Delta x)^2}(S_{i+1}^1-2S_i^1+S_{i-1}^1)-K_1\Delta t\frac{S_i^1M_i^1}{K_4+S_i^1}+\Delta tF_i^1,\\
&M_i^2=M_i^1-K_2\Delta tM_i^1+K_3\Delta t\frac{S_i^1M_i^1}{K_4+S_i^1}+\Delta tG_i^1\\
&\ +\frac{d_2\Delta t}{(\Delta x)^2}\left[\lambda\left(\frac{M_{i+1}^1+M_i^1}{2}\right)(M_{i+1}^1-M_i^1)-\lambda\left(\frac{M_{i}^1+M_{i-1}^1}{2}\right)(M_{i}^1-M_{i-1}^1)\right],
\end{align*}
and for $n=\overline{2,N-1}$,
\begin{align*}
&\frac{S_i^{n+1}-S_i^{n-1}}{2\Delta t}=\frac{d_1}{(\Delta x)^2}(\hat{S}_{i+1}^n-2\hat{S}_i^n+\hat{S}_{i-1}^n)-K_1\frac{S_i^nM_i^n}{K_4+S_i^n}+F_i^n,\\
&\frac{M_i^{n+1}-M_i^{n-1}}{2\Delta t}=-K_2M_i^n+K_3\frac{S_i^nM_i^n}{K_4+S_i^n}+G_i^n\\
&\quad +\frac{d_2}{(\Delta x)^2}\left[\lambda\left(\frac{M_{i+1}^n+M_i^n}{2}\right)(\hat{M}_{i+1}^n-\hat{M}_{i}^n)-\lambda\left(\frac{M_{i}^n+M_{i-1}^n}{2}\right)(\hat{M}_{i}^n-\hat{M}_{i-1}^n)\right],
\end{align*}
where 
$$\hat{S}_i^n:=\frac{S_i^{n+1}+S_i^n+S_i^{n-1}}{3},\quad \hat{M}_i^n:=\frac{M_i^{n+1}+M_i^n+M_i^{n-1}}{3}.$$
Denoting 
$$\alpha:=\frac{2\Delta td_1}{3(\Delta x)^2},\quad\lambda_i^n:=\frac{2\Delta td_2}{3(\Delta x)^2}\lambda\left(\frac{M_{i}^n+M_{i-1}^n}{2}\right),$$
we obtain
\begin{align}
&-\alpha S_{i-1}^{n+1}+(1+2\alpha)S_i^{n+1}-\alpha S_{i+1}^{n+1}=f_{i+1}^n,\label{eq23}\\
&-\lambda_i^nM_{i-1}^{n+1}+(1+\lambda_i^n+\lambda_{i+1}^n)M_i^{n+1}-\lambda_{i+1}^nM_{i+1}^{n+1}=g_{i+1}^n,\label{eq24}
\end{align}
where
\begin{align*}
f_{i+1}^n=&\alpha(S_{i+1}^n+S_{i+1}^{n-1}-2S_i^n-2S_i^{n-1}+S_{i-1}^n+S_{i-1}^{n-1})+S_i^{n-1}\\
&-2\Delta tK_1\frac{S_i^nM_i^n}{K_4+S_i^n}+2\Delta tF_i^n,\\
g_{i+1}^n=&\lambda_{i+1}^n(M_{i+1}^n+M_{i+1}^{n-1}-M_i^n-M_i^{n-1})-\lambda_{i}^n(M_{i}^n+M_{i}^{n-1}-M_{i-1}^n-M_{i-1}^{n-1})\\
&+M_i^{n-1}-2\Delta tK_2M_i^n+2\Delta tK_3\frac{S_i^nM_i^n}{K_4+M_i^n}+2\Delta tG_i^n.
\end{align*}
We can rewrite the equations \eqref{eq23} and \eqref{eq24} in the following matrix forms:
\begin{align}
\label{eq:matrix_form}
\mathbf{A}\mathbf{S}^{n+1}=\mathbf{f}^n, \quad \mathbf{B}^n\mathbf{M}^{n+1}=\mathbf{g}^n, \quad n=\overline{2,N-1}.
\end{align}
Here, $\mathbf{A}$ and $\mathbf{B}^n$ are $(I-2)\times(I-2)$ symmetric matrices given by
\begin{equation*}
\mathbf{A}=\begin{bmatrix}
1+2\alpha & -\alpha & 0 &\cdots &0 &0 &0\\
-\alpha & 1+2\alpha & -\alpha & \cdots &0 &0 &0\\
\vdots &\vdots &\vdots &\ddots &\vdots &\vdots &\vdots\\
0& 0& 0& \cdots &-\alpha & 1+2\alpha & -\alpha\\
0& 0& 0& \cdots &0 & -\alpha & 1+2\alpha
\end{bmatrix},
\end{equation*}

\begin{equation*}
\mathbf{B}^n=\begin{bmatrix}
1+\lambda_2^n+\lambda_3^n & -\lambda_3^n & 0 &\cdots &0 &0 &0\\
-\lambda_3^n & 1+\lambda_3^n+\lambda_4^n & -\lambda_4^n & \cdots &0 &0 &0\\
\vdots &\vdots &\vdots &\ddots &\vdots &\vdots &\vdots\\
0& 0& 0& \cdots &-\lambda_{I-2}^n & 1+\lambda_{I-2}^n+\lambda_{I-1}^n & -\lambda_{I-1}^n\\
0& 0& 0& \cdots &0 & -\lambda_{I-1}^n & 1+\lambda_{I-1}^n+\lambda_{I}^n
\end{bmatrix}.
\end{equation*}
In the first equation of~\eqref{eq:matrix_form}, 
$$\mathbf{S}^{n} = \left[ S_{2}^{n},\cdots,S_{I-1}^{n} \right]^{\mathrm{T}} \; 
\mbox{and} \; \mathbf{f}^n=\left[f_3^n+\alpha\mu_1^{n+1}, f_4^n,\cdots, f_{I-1}^n,f_{I}^n + \alpha \mu_{2}^{n+1}\right]^{\mathrm{T}},$$
and in the second one, 
$$\mathbf{M}^{n} = \left[ M_{2}^{n},\cdots,M_{I-1}^{n} \right]^{\mathrm{T}} \; 
\mbox{and} \; \mathbf{g}^n=\left[g_{3}^{n} + \lambda_{2}^{n} \mu_{3}^{n+1}, g_{4}^{n},\cdots, g_{I-1}^{n},g_{I}^{n} + \lambda_{I}^{n} \mu_{4}^{n+1} \right]^{\mathrm{T}}.$$

\end{document}